\documentclass[a4paper,12pt]{scrartcl}
\usepackage[english]{babel}
\usepackage[T1]{fontenc}
\usepackage[utf8]{inputenc}
\DeclareUnicodeCharacter{00A0}{ }
\DeclareUnicodeCharacter{014C}{\=O}
\usepackage{mathptmx}
\usepackage[scaled=.92]{helvet}
\usepackage{courier}
\usepackage{paralist}
\usepackage[nodetail]{optional}

\usepackage[fixlanguage]{babelbib}
\selectbiblanguage{english}

\newcommand{\Rn}[1]{{\romannumeral #1}}
\newcommand{\RN}[1]{\uppercase\expandafter{\romannumeral #1\relax}}

\usepackage{amsmath,amsfonts,amssymb}
\usepackage[mathscr]{euscript}

\newcommand{\Z}{\mathbb{Z}}
\newcommand{\R}{\mathbb{R}}
\newcommand{\C}{\mathbb{C}}

\newcommand{\RE}{\operatorname{Re}}

\newcommand{\eps}{\varepsilon}
\newcommand{\intd}{\mspace{0.3mu}\operatorname{d}\mspace{-2.0mu}}
\newcommand{\supp}{\operatorname{supp}}
\newcommand{\ord}{\operatorname{ord}}
\newcommand{\id}{\operatorname{id}}

\newcommand{\Lloc}{L_{\operatorname{loc}}} 
\newcommand{\Cinfty}{\ensuremath{C^{\infty}}}
\newcommand{\Analytic}{{\mathcal A}} 
\newcommand{\Aprime}{{\mathcal A}'} 
\newcommand{\Dprime}{{\ensuremath{\mathcal{D}^\prime}}} 
\newcommand{\Sob}[1]{\ensuremath{H^{#1}}} 
\newcommand{\V}{\ensuremath{\mathcal V}} 
\newcommand{\Hfcn}{\ensuremath{\mathcal B}} 
\newcommand{\Eigen}{\ensuremath{\mathcal E}} 
\newcommand{\mEigen}{\ensuremath{\mathcal E^*}} 
\newcommand{\Diff}{\operatorname{Diff}} 
\newcommand{\D}{{\mathbb D}} 

\newcommand{\lie}{\mathfrak} 
\newcommand{\astarc}{\ensuremath{\lie{a}_{\C}^*}}
\newcommand{\Poisson}{{\mathcal P}}  
\newcommand{\bv}{\beta}  
\newcommand{\sfcn}{\phi}  
\newcommand{\cfcn}{\mathbf{c}} 
\newcommand{\efcn}{\mathbf{e}} 
\newcommand{\OOplus}{O_+} 
\newcommand{\OO}{O} 
\newcommand{\OOplusclosure}{\bar O_+} 
\newcommand{\dd}[1]{\partial_{#1}} 
\newcommand{\tdt}{t\partial_t} 
\newcommand{\ipol}{\iota} 

\renewcommand{\a}{\lie{a}}  

\usepackage{amsthm}
\newtheorem{theorem}{Theorem}[section]
\newtheorem{proposition}[theorem]{Proposition}
\newtheorem{corollary}[theorem]{Corollary}
\newtheorem{lemma}[theorem]{Lemma}
\theoremstyle{definition}

\newtheorem{remark}[theorem]{Remark}

\usepackage[pdftex]{color}
\usepackage[pdftex]{graphicx}
\usepackage[pdftex,
	colorlinks=true,
	linkcolor=gruen,
        urlcolor=blue,
        citecolor=blue,
	bookmarksopen=false,
	pdfpagemode=None]{hyperref}
\definecolor{gruen}{rgb}{0,0.5,0}

\title{Boundary Values of Eigenfunctions on \\ Riemannian Symmetric Spaces}
\author{S.~Hansen, J.~Hilgert, and A.~Parthasarathy}
\date{July 18, 2018}
\makeindex
\begin{document}

\maketitle
\begin{abstract}
We give a new and self-contained proof of a generic version of the (former) Helgason conjecture. It says that for generic spectral parameters the Poisson transform is a topological isomorphism, with the inverse given by a boundary value map.
Following Oshima's approach to a simplified definition of boundary values, and using the earlier work of Baouendi and Goulaouic on Fuchsian type equations, our proof is along the lines of our earlier work in the rank one distributional case, and works for both the hyperfunction and the distribution setting. 
\end{abstract}

\section{Introduction} 

In \cite{Helgason74conjecture} Helgason proved that for Riemannian symmetric spaces $G/K$ of rank one all joint eigenfunctions under the action of the algebra of invariant differential operators are Poisson integrals of their hyperfunction boundary values, except for an explicitly given set of spectral parameters determining the eigenvalues. In the same paper he conjectured that this theorem can be extended to higher rank spaces replacing the (geodesic) boundary by the Furstenberg boundary. In    
\cite{Helgason76duality2} he proved such an extension for $K$-finite functions. 

The proof of the conjecture given in \cite{Kashiwara78eigen} makes essential use of the theory of boundary values for systems of differential equations with regular singularities developed by Kashiwara and \=Oshima in \cite{KashiwaraOshima77}, which in turn is based on the theory of microdifferential operators as laid out in \cite{SatoKawaiKashiwara73}. A beautiful exposition, but without full proofs, of the main results as well as their historical evolution is given in \cite{Schlichtkrull84hyperf} (see in particular Chapter 5).

The question which joint eigenfunctions are Poisson integrals of distributions on the Furstenberg boundary was settled for rank one spaces in \cite{Lewis78eigenf}, and in general by \=Oshima and Sekiguchi in \cite{OshimaSekiguchi80}. The latter again depended on \cite{KashiwaraOshima77}.
Other authors searched for alternative proofs using methods more traditional in noncommutative harmonic analysis. Such proofs, based on asymptotic expansions of eigenfunctions, were given in \cite{Wallach83} and  \cite{BanSchlichtkrull87}, where the latter depends on the former, but also extends it by giving an explicit description of the inverse of the Poisson transform in terms of leading coefficients.

Yet another approach allowing for further generalization was initiated by Schmid in \cite{Schmid85}, where he sketched an outline of how Helgason's conjecture could be derived from a characterization of maximal globalizations of Harish-Chandra modules. The proof of the latter that was sketched in  \cite{KashiwaraSchmid94} and worked out in \cite{kashiwara08equivariant} is based on equivariant derived categories. To our understanding it does not give an explicit description of the boundary value map.

In this paper we give a complete and self-contained proof that for generic spectral parameters the Poisson transform is a topological isomorphism with inverse a boundary value map. 
The proof works for the hyperfunction and distribution cases alike and follows the lines of the proof we gave in \cite{SHJHAP17resonscattpoles} for the rank one case. It does not use microlocal methods and gives an explicit description of the boundary value map which is much simpler than the one given in \cite{Kashiwara78eigen} based on \cite{KashiwaraOshima77}. We rather follow the approach to Fuchsian type differential equations presented in \cite{Oshima83bdryval}, in which the author already has announced that the methods presented would be sufficient to prove the Helgason conjecture.
In fact, we hope that the full conjecture,
including the characterization of the exceptional set of spectral parameters as the zeros of the Harish-Chandra $\mathbf e$-function,
can be deduced, similarly to \cite{Kashiwara78eigen},
from our generic theorem by first holomorphically extending the boundary value map to at least one spectral parameter
from each Weyl group orbit, and then applying intertwining operators.

We conclude this introduction by explaining our results in some detail.
As is well-known, in compactifications which embed the symmetric space as a manifold with corner,
invariant differential operators are regular singular with respect to the edge of the corner.
The edge is the Furstenberg boundary of the symmetric space.
The solution theory of these operators and of Fuchsian type operators associated with these,
is controlled by their characteristic exponents.
The set of characteristic exponents is an affine image of the Weyl group orbit of the underlying spectral parameter.
Boundary values of joint eigenfunctions are associated to characteristic exponents.
If characteristic exponents with negative integer components are avoided, then one has unique solvability
of Fuchsian equations with data supported in the respective boundary faces.
As in \cite{KashiwaraOshima77} and \cite{Oshima83bdryval},
joint eigenfunction are, after a shift by the characteristic exponent, and by solving such Fuchsian type equations,
extended as (shifted) joint eigenfunctions around the edge.
The theory of \cite{BaouendiGoulaouic73} is sufficient for solving the Fuchsian type equations;
appendix~\ref{section-Fuchsian-eqns} contains a detailed exposition.
The associated Fuchsian type operators act tangentially at the faces of the corner.
From these is constructed a differential operator which acts transversally to each face.
Applying this operator to the extended eigenfunction, we obtain, up to a multiplicative constant,
the (tensor) product of its boundary value with the Dirac function of the edge.
The constant is non-zero under the condition that the components of the characteristic exponent are simple.
Assuming this condition and the condition on absence of negative integer components mentioned above,
the boundary value map is defined.
The boundary value map inverts the Poisson transform up to the Harish-Chandra $\mathbf c$-function.      
This assertion is our main result, see Theorem~\ref{theorem-kkmoot}.

\section{Preliminaries} 
\label{sect-Prelims}

Let $X=G/K$ be a Riemannian symmetric space of noncompact type with 
$G$, a noncompact connected real semisimple Lie group with finite centre, and $K$, a maximal compact subgroup of $G$. 
Fix an Iwasawa decomposition $G=KAN$ thus making the following choices:
A Cartan decomposition $\lie{g}=\lie{k}\oplus\lie{p}$ of the Lie algebra of $G$,
a maximal abelian subalgebra $\lie a$ of $\lie p$.
The Iwasawa projections $\kappa:G\to K$ and $H:G\to \lie a$ are given by $g\in \kappa(g)\exp(H(g)) N$.
The dimension $\ell=\dim_\R \lie a$ is the rank of $X$,
$\lie a^*$ and $\astarc$ are the dual of $\lie a$ and its complexification.
Fix a positive system $\Sigma^+$ of the set $\Sigma\subset\lie a^*$ of restricted roots.
We write $\rho\in\lie a^*$ for the weighted half-sum of positive roots,
$W$ for the Weyl group for the restricted roots.
Let $M$ denote the centralizer of $\a$ in $K$.
The compact space $B=K/M$ is called the Furstenberg boundary of $X$.

The differential operators on $X$ which are invariant under left translation by
elements of $G$ form a commutative algebra, the algebra $\D(X)$ of invariant differential operators;
see \cite[Chapter II]{Helgason84GGA}.
By the Harish-Chandra isomorphism, $\D(X)$ is isomorphic to the algebra $I(\astarc)$
of $W$-invariant polynomials on $\astarc$.
Every character of $I(\astarc)$ is given by evaluation at some point of $\astarc$.
Thus, every character of $\D(X)$ is, for some $\lambda\in\astarc$, of the form $\chi_\lambda$,
where $\mu\mapsto\chi_\mu(D)$ is the polynomial in $I(\astarc)$ which
corresponds to $D$ under the Harish-Chandra isomorphism.
The equality $\chi_\lambda=\chi_\mu$ holds if and only if $\lambda$ and $\mu$ belong to the same $W$-orbit.
The system of diffential equations
\begin{equation}
\label{eq-joint-eigensys}
Du=\chi_{\lambda}(D)u,\quad D\in\D(X),
\end{equation}
is called the joint eigensystem with spectral parameter $\lambda$.
The solutions $u$ of \eqref{eq-joint-eigensys} are called joint eigenfunctions,
and the space of joint eigenfunctions is denoted $\Eigen_\lambda(X)$.
The Laplace--Beltrami operator belongs to $\D(X)$ and is elliptic.
By elliptic regularity, $\Eigen_\lambda(X)$ is a subspace of the space $\Analytic(X)$ of analytic functions on $X$.
Here and in the following, analytic means real analytic.
Observe that $\Eigen_{w\cdot\lambda}(X)= \Eigen_{\lambda}(X)$ for $w\in W$.
The left regular representation $T_\lambda$ on $\Eigen_\lambda(X)$ coming from the action of $G$ on $X$
by left translation is called the eigenspace representation with spectral parameter $\lambda$.

The minimal parabolic $P=MAN$ is associated to the Iwasawa decomposition.
The homogeneous space $G/P$ is $G$-equivariantly diffeomorphic to the
Furstenberg boundary $B$ by the map $kM\mapsto kP$,
$k\in K$, where the $G$-action on $B$ is given by $g\cdot kM=\kappa(gk)M$.

We set $a^\mu=e^{\mu(\log a)}$ for $a\in A$, $\mu\in\astarc$.
Given $\lambda \in \astarc$, the character $man\mapsto a^{\rho-\lambda}$ defines a one-dimensional
representation $\C_\lambda$ of $P$.
Let $L_\lambda=G\times_P\C_\lambda$ be the associated homogeneous line bundle over $G/P$.
By inducing from $\C_\lambda$, we obtain the
spherical principal series representation $\pi_\lambda$ of $G$ on the space $\Hfcn(G/P; L_\lambda)$
of hyperfunction sections.
See \cite{SatoKawaiKashiwara73}, \cite[Ch.~\RN 9]{Hormander90anaOne}, and \cite{Schlichtkrull84hyperf} for the theory of hyperfunctions.
The sections in $L_\lambda$ are canonically identified with hyperfunctions $\varphi$ on $G$
which satisfy the equivariance condition  $\varphi(gman)=a^{\lambda-\rho}\varphi(g)$ for $g\in G$, $man\in P$.
By the Iwasawa decomposition, restriction of $\varphi$ to $K$ gives an isomorphism $\Hfcn(G/P; L_\lambda) \equiv \Hfcn(B)$.
The inverse assigns to a hyperfunction $f$ on $B$ the section $\varphi$ with
$\varphi(kan)=a^{\lambda-\rho}f(kM)$, $a\in A$ and $n\in N$.
We then have a realization of the spherical principal series representation $\pi_\lambda$ on $\Hfcn(B)$,
which is given explicitly by
\begin{equation*}
\label{eq-spher-rep-on-Aprime-B}
(\pi_\lambda(g) f)(kM) = e^{-(\rho-\lambda)(H(g^{-1}k))} f(\kappa(g^{-1}k)M).
\end{equation*}
See \cite[Ch.~\RN 7 \S 1]{Knapp86repexa}, and \cite[Ch.~\RN 6 \S 3 (13)]{Helgason94GASS}.

The Poisson transformation $\Poisson_\lambda$ with spectral parameter $\lambda\in\astarc$ is a $G$-equivariant map
which intertwines $\pi_\lambda$ with $T_\lambda$:
\[
\Poisson_\lambda: \Hfcn(G/P; L_\lambda) \to \Eigen_\lambda(X),
\quad
(\Poisson_\lambda f)(gK) = \int_K f(gk)\intd k.
\]
Alternatively, in the compact picture of $\pi_\lambda$, $\Poisson_\lambda:\Hfcn(B)\to \Eigen_\lambda(X)$,
\begin{equation*}
\label{def-Poisson-transform}
(\Poisson_\lambda f)(x)= \int_B e^{(\rho+\lambda)(A(x,b))} f(b)\intd b.
\end{equation*}
Here $A$ denotes the horocycle bracket $A(x,b)=-H(g^{-1}k)\in\a$ of $x=gK\in X$, $b=kM\in B$.
Since $B$ is a closed analytic manifold, $\Hfcn(B)$ naturally identifies with the space
$\Aprime(B)$ of analytic functionals on $B$.
Invoking the closed graph theorem for Fr\'echet spaces, it follows
that $\Poisson_\lambda:\Aprime(B)\to\Eigen_\lambda(X)$ is a continuous linear operator.

The spherical function $\sfcn_\lambda =\Poisson_\lambda 1$ is the unique $K$-invariant
element of $\Eigen_\lambda(X)$ with the normalisation $\sfcn(o)=1$.
Moreover, $\sfcn_{w\cdot\lambda}=\sfcn_\lambda$ holds for $w\in W$ and $\lambda\in\astarc$.
The behaviour at infinity of the spherical function $\phi_\lambda$, for sufficiently regular $\lambda$,
is captured by the Harish-Chandra $\cfcn$-function.
The $\cfcn$-function has an integral representation given by
$\cfcn_\lambda=\int_{\bar N}e^{-(\lambda+\rho)(H(\bar n))}\intd \bar n$
for $\RE\lambda\in \a_+^*$ with the measure $d\bar n$ suitably normalised,
where we wrote $\bar N$ for the image of the nilpotent subgroup $N$  under the Cartan involution.

The simple roots $\{\alpha_1,\ldots,\alpha_\ell\}\subset \Sigma^+$ are a basis of $\lie a^*$.
Let $\{H_1,\ldots,H_\ell\}\subset\lie a$ be the dual basis,
that is $\alpha_j(H_k)=\delta_{jk}$ holds for $1\leq j,k\leq \ell$.
We identify $\astarc$ with $\C^\ell$ via the basis $\alpha_1,\ldots,\alpha_\ell$, i.e.,
$\astarc\ni\zeta = \sum\nolimits _j \zeta_j\alpha_j = \big(\zeta(H_1),\ldots, \zeta(H_\ell)\big)\in \C^\ell$.
Let $\a^+$ be the positive Weyl chamber in $\a$ determined by $\Sigma^+$, and 
$\a^*_+$ the dual cone of $\a^+$ in $\a^*$.
Set $A^+=\exp\a^+$.
We then identify $A^+$ with the open octant $]0,1[^\ell$ by the diffeomorphism
$t=(t_1, \ldots , t_\ell)\mapsto a=\exp(-\sum\nolimits _j (\log t_j)H_j)$.
From the decomposition $G=K\overline{A^+}K$, we have the diffeomorphism $(kM,a)\mapsto ka\cdot o$
from $B\times A^+$ onto the open dense subset $KA^+\cdot o$ of $X$.
Define the octants $\OOplus=B\times ]0,1[^\ell$ and $\OO=B\times ]-1, 1[^\ell$.
The map
\begin{equation}
\label{eq-diffeom-to-corner}
\OOplus\to KA^+\cdot o, \quad (kM,t)\mapsto ka\cdot o, \quad a=\exp\big(-\sum\nolimits _j \log(t_j) H_j\big),
\end{equation}
is an analytic diffeomorphism from $\OOplus$ onto a dense open subset of $X$.
We call the closure $\OOplusclosure$ of $\OOplus$ in $\OO$ the corner domain,
and $B\equiv \{t=0\}\subset\OOplusclosure$ its edge.
A local diffeomorphism at the edge of the corner is an analytic map
\begin{equation}
\label{eq-Phi-edge-diffeom}
\Phi:(z,t)\mapsto(z',t'), \quad\text{where $t_j'=\tau_j(z,t)t_j$, $0<\tau_j$ analytic,}
\end{equation}
which is a diffeomorphism between connected open subsets of $\OO$ which intersect the edge.
(We reserve the notation $t$ for the special coordinates defined in \eqref{eq-diffeom-to-corner}.)
Let $g\in G$.
By \cite[Lemma 4.1-2]{Kashiwara78eigen} or \cite[Lemma~4.1.6]{Schlichtkrull84hyperf}
there exists a diffeomorphism $\Phi=\Phi_g$ between open neighbourhoods of the edge,
such that \eqref{eq-Phi-edge-diffeom} and
\begin{equation}
\label{eq-Phi-g}
\Phi_g(x)=g^{-1}\cdot x, \quad
\Phi_g(z,0)=(\kappa(g^{-1}k)M,0), \quad
\tau_j(z,0)=e^{\alpha_j(A(g\cdot o,z))}
\end{equation}
hold for $x\in KA^+\cdot o\subset X$, $z=kM\in B$.

The conormal bundle of the edge, $N^*B$, is spanned by the differentials $\intd t_j$.
Define the line bundle $(N^*B)^\sigma\to B$, $\sigma=(\sigma_1,\ldots,\sigma_\ell)\in\C^\ell$, by the local frames
$(\intd t')^\sigma=\prod_j (\intd t_j')^{\sigma_j}$ and the transition laws
\begin{equation}
\label{eq-trafo-NstarBsigma}
(\intd t')^\sigma = \tau(z,0)^\sigma (\intd t)^\sigma.
\end{equation}
Note that $(N^*B)^\sigma$ is trivialized by the global section $(\intd t)^\sigma$.
Using \eqref{eq-Phi-g}, $G$ acts equivariantly on $(N^*B)^\sigma\to B$.
The line bundles $(N^*B)^{\rho-\lambda}\to B$ and $L_\lambda\to G/P$ are $G$-isomorphic
as $G$-equivariant line bundles; see \cite[Lemma 5.2.1]{Schlichtkrull84hyperf}.
So, $\Hfcn(G/P;L_\lambda)$ and $\Hfcn(B;(N^*B)^{\sigma})$ are $G$-isomorphic representation spaces.
In particular, $\Poisson_\lambda:\Hfcn(B;(N^*B)^{\sigma})\to\Eigen_\lambda(X)$ is a $G$-homomorphism.

Let $d_X(x)$ denote the Riemannian distance from $x\in X$ to the origin $o=K\in X$.
A function $u$ on the symmetric space $X$ is said to be of slow growth
if $ue^{-rd_X}$ is a bounded function for some $r>0$ \cite[Section~I.2]{BanSchlichtkrull87}.
Denote by $\mEigen_{\lambda}(X)$ the space of $u\in\Eigen_\lambda(X)$ which are of slow growth.
By \cite[Lemma 2.1 (\Rn{3}), (\Rn{4})]{BanSchlichtkrull87}, $u$ is of slow growth iff
there exists $\nu\in\mathbb N^\ell$ such that $t^\nu u$ is bounded on $\OOplus$.
It was proved in \cite{Lewis78eigenf} that $\Poisson_\lambda$ maps $\Dprime(B)$ into $\mEigen_\lambda(X)$.
As in \cite{SHJHAP17resonscattpoles}, we use Sobolev spaces of negative integer order,
$\Sob{-m}(\OOplus)$, to deal with the growth condition.
By (the proof of) \cite[Lemma 2.1]{SHJHAP17resonscattpoles} the union $\cup_m \Sob{-m}(\OOplus)$
contains the functions of slow growth.
In particular, functions of slow growth are extendible distributions,
Since differential operators are continuous maps into the space of distributions,
$\Eigen_\lambda(X)\cap\Sob{-m}(\OOplus)$ is a closed subspace of $\Sob{-m}(\OOplus)$.
The space $\Dprime(B)=\cup_m \Sob{-m}(B)$, and the space
\begin{equation}
\label{eq-eigen-slowgrowth}
\mEigen_\lambda(X) = \cup_{m\in\mathbb N} \big(\Eigen_\lambda(X)\cap\Sob{-m}(\OOplus)\big)
\end{equation}
are (DFS)-spaces, that is strong duals of Fr\'echet--Schwartz spaces;
\cite[Prop.~25.20]{MeiseVogt97fa}.
Invoking the closed graph theorem for (DFS)-spaces,
the Poisson transform is seen to be a continuous linear operator $\Poisson_\lambda:\Dprime(B)\to \mEigen_\lambda(X)$.
For every $m\in\mathbb N$ and every bounded domain $\Lambda\subset\astarc$ there exists $m'\in\mathbb N$
such that $\Lambda\ni\lambda\to \Poisson_\lambda$ is a holomorphic map into the
Banach space $\mathcal L\big(\Sob{-m}(B), \Sob{-m'}(\OOplus)\big)$ of bounded linear operators.

\section{Differential operators}

Conjugation with the pullback by the diffeomorphism \eqref{eq-diffeom-to-corner}
identifies $\D(X)$ with a subalgebra of the algebra of differential operators with analytic coefficients
on $\OOplus$, $\Diff(\OOplus)$,
By restricting coefficients from $\OO$ to $\OOplus$ we also identify $\Diff(\OO)$
with a subalgebra of $\Diff(\OOplus)$.

The vector field $\theta_j:=t_j\partial_{t_j}$ is tangent to the $j$-th wall, $\OO\cap\{t_j=0\}$.
Denote by $\Diff_b(\OO)$ the subalgebra of $\Diff(\OO)$ which is generated by the
vector fields tangent to each wall.
If $z_1,\ldots,z_n$ are local coordinates of $B$, then the tangent vector fields are spanned,
over the ring of analytic functions, by the vector fields $\dd{z_i}$ and $\theta_j$.
Notice that, for $s=(s_1,\ldots,s_\ell)\in\C^\ell$,
\[ t^{-s}\dd{z_i}t^s=0, \quad t^{-s}\theta_jt^s=\theta_j+s_j, \]
hold in $\Diff(\OOplus)$.
Here $t^s=t_1^{s_1}\cdots t_{\ell}^{s_\ell}$.
It follows that, for every $D\in\Diff_b(\OO)$ the conjugate operator $t^{-s}Dt^s$,
initially a differential operator on $\OOplus$,
uniquely extends to an element of $\Diff_b(\OO)$ which we denote by $D^s$.

The vector fields on $\OO$ which vanish in the edge $B=\{t=0\}$
and which are tangent to each wall form a Lie algebra, $\V_0(\OO)$.
Denote by $\Diff_0(\OO)$ the subalgebra of $\Diff_b(\OO)$ which is generated by $\V_0(\OO)$.
Using local coordinates $z_i$ on $B$ every $D\in\Diff_0(\OO)$ can be written in the form
\begin{equation}
\label{eq-D-Diffzero-local}
D= \sum\nolimits _{p,q} a_{p,q} (t\dd{z})^q \theta^p,
\quad t\dd{z}:=(t_j\dd{z_i})_{i,j},
\quad \theta:=(\theta_1,\ldots,\theta_{\ell}).
\end{equation}
Here the sum is finite, $p$ and $q$ are multiindices of non-negative integers,
and the coefficients $a_{p,q}$ are analytic functions.
Forming a (tensor) product of classical homogeneous distributions,
we put $t_+^s=(t_1)_+^{s_1}\ldots(t_\ell)_+^{s_\ell}\in\Dprime(\OO)$.
It follows from \eqref{eq-D-Diffzero-local} that $D\in\Diff_0(\OO)$ maps distributions
of the form $at_+^s$, $a$ analytic, into distributions of the same form.
The most singular term of $D(at_+^s)=bt_+^s$ is determined algebraically as $b|_{t=0} =\ipol(D)(s) a|_{t=0}$.
Here
\begin{equation}
\label{eq-ipol-D-defined}
\ipol(D)(s) := (D^s 1)|_{t=0}\in \Analytic(B)[s]
\end{equation}
is a polynomial in $s\in\C^\ell$, the indicial polynomial of $D$ with respect to the edge $B$.

\begin{proposition}
\label{prop-invdo-subalg-Diffzero}
$\D(X)\subset\Diff_0(\OO)$, and
$\ipol(D)(\sigma) = \chi_{\rho-\sigma}(D)$ for $\sigma\in\astarc$, $D\in\D(X)$.
\end{proposition}
\begin{proof}
	
	We know, by \cite[Theorem 4.3.1]{Schlichtkrull84hyperf}, that in the coordinates of the Oshima compactification (see \cite[Chapter 4]{Schlichtkrull84hyperf} for a detailed description of the compactification), an invariant differential operator $D\in\D(X)$ can be written in the form \eqref{eq-D-Diffzero-local}. Then we invoke \cite[Proposition 11]{Oshima78realization} to observe that $\OO$ is naturally embedded as an open and dense subset, containing the boundary $B$, of the Oshima compactification. This shows the first assertion of the Proposition. The second assertion then follows exactly as in the proof of \cite[Theorem 4.3.1]{Schlichtkrull84hyperf}. 
\end{proof}
	The fact that indicial polynomials of invariant differential operators are given, as in Proposition~\ref{prop-invdo-subalg-Diffzero}, by certain $W$-invariant polynomials on $\astarc$ shows, in particular, that indicial polynomials of invariant differential operators have constant coefficients.
We also remark here that, as proved in \cite[Theorem 4.3.1]{Schlichtkrull84hyperf},
it follows from Proposition~\ref{prop-invdo-subalg-Diffzero} 
and certain properties of the algebra $I(\astarc)$
that the joint eigensystems \eqref{eq-joint-eigensys}
are regular singular in the sense of \cite[Definition 3.2]{Kashiwara78eigen}.
By definition, $\sigma\in\C^\ell$ is called a characteristic exponent of \eqref{eq-joint-eigensys}
iff $\ipol(D-\chi_{\lambda}(D))(\sigma)=0$ holds for all $D\in\D(X)$.
By Proposition~\ref{prop-invdo-subalg-Diffzero}
the characteristic exponents of \eqref{eq-joint-eigensys} are $\rho- w\cdot\lambda$, $w\in W$.

To analyze the behaviour of joint eigenfunctions at the edge $B$,
we shall use the fact that these satisfy Fuchsian type equations with respect to the walls.
See Appendix~\ref{section-Fuchsian-eqns} for the definition of Fuchsian type differential operators.
If $D$ is a differential operator which is of Fuchsian type with respect to the $j$-th wall, then
\[ \ipol_j(D)(s):= (D^{s\alpha_j} 1)|_{t_j=0}, \quad s\in\C, \]
is its indicial polynomial.
\begin{proposition}
\label{prop-weak-reg-sing}
Let $1\leq j\leq\ell$.
Set $m_j=|W/W_j|$ where $W_j$ is the subgroup of $W$ which leaves $H_j$ fixed.
There exist $D_{jk}\in\D(X)$ of order $\leq k$, and $R_j\in\Diff_b(\OO)$ such that, for $\lambda\in\astarc$,
\begin{equation}
\label{eq-def-Pj}
\begin{aligned}
P_j(\lambda) &:= \sum\nolimits _{k=1}^{m_j} \theta_j^{m_j-k}\big(D_{jk}-\chi_\lambda(D_{jk})\big) \\
    &\phantom{:}= \theta_j^{m_j} - \sum\nolimits _{k=1}^{m_j} \chi_\lambda(D_{jk}) \theta_j^{m_j-k}-t_jR_j.
\end{aligned}
\end{equation}
In particular, $P_j(\lambda)$ is of Fuchsian type with respect to the $j$-th wall, and
\[ p_{j,\lambda}(s):=\ipol_j(P_j(\lambda))(s) = s^{m_j} - \sum\nolimits _{k=1}^{m_j} \chi_\lambda(D_{jk}) s^{m_j-k} \]
is the indicial polynomial.
\end{proposition}
\begin{proof}
See Proposition 4.3.2 and formula (4.23) in \cite{Schlichtkrull84hyperf}.
\end{proof}

The polynomials $p_{j,\lambda}$ and $\ipol(D)$, $D\in\D(X)$, have constant coefficients.
So, for $\sigma\in\C^\ell$, the first equality of
\[
\ipol_j(P_j(\lambda))(\sigma_j) = \ipol(P_j(\lambda))(\sigma)
     = \sum\nolimits _{k=1}^{m_j} \ipol\big(D_{jk}-\chi_\lambda(D_{jk})\big)(\sigma) \sigma_j^{m_j-k}
\]
holds meaningfully in $\C$, the second follows from \eqref{eq-def-Pj}.
The roots of the indicial equation $p_{j,\lambda}(s)=0$ are called the
characteristic exponents of $P_j(\lambda)$ with respect to the $j$-th wall.
The displayed formula implies that the $j$-th components of the characteristic exponents
of the eigensystem \eqref{eq-joint-eigensys} are characteristic exponents of $P_j(\lambda)$.
\begin{corollary}
\label{cor-charexp-Pj}
The characteristic exponents of $P_j(\lambda)$ are $(\rho-w\cdot\lambda)(H_j)$, $w\in W$.
The characteristic exponent $(\rho-\lambda)(H_j)$ is simple 
iff $\lambda(H_j-w\cdot H_j)\neq 0$ holds for every $w\in W\setminus W_j$.
\end{corollary}
\begin{proof}
The first assertion of the corollary states that
\begin{equation}
\label{eq-root-inclusion-for-pj}
\{\rho(H_j)-\lambda(w\cdot H_j)\mid w\in W\} = \{s\in\C\mid p_{j,\lambda}(s)=0\}
\end{equation}
holds.
We already know that the left-hand side is contained in the right-hand side.
Define $\Lambda$ as the set of $\lambda$'s which satisfy $\lambda(w\cdot H_j-w'\cdot H_j)\neq 0$
whenever $w, w'\in W$ and $w^{-1}w'\not\in W_j$.
If $\lambda\in\Lambda$, then the cardinalities of the sets on
either side of \eqref{eq-root-inclusion-for-pj} are equal to $m_j$.
So, for $\lambda\in\Lambda$, \eqref{eq-root-inclusion-for-pj} holds,
and all roots of the polynomial $p_{j,\lambda}$ are simple.
The set $\Lambda$ is open and dense in $\astarc$.
The roots of the monic polynomials $p_{j,\lambda}$ depend continuously on the coefficients, hence on $\lambda$.
It follows that \eqref{eq-root-inclusion-for-pj} holds for every $\lambda\in\astarc$.

To prove the second assertion of the corollary, we assume there exists $\eps>0$
such that $|\lambda(H_j-w\cdot H_j)|>2\eps$ if $w\in W\setminus W_j$.
We have to show that the root $s_0=(\rho-\lambda)(H_j)$ of $p_{j,\lambda}$ is simple.
Choose an open neighbourhood $U\subset\astarc$ of $\lambda$
such that $|(\lambda-\mu)(H_j-w\cdot H_j)|<\eps$ holds for $w\in W$ and $\mu\in U$.
Hence,
$|s_0-(\rho(H_j)-\mu(w\cdot H_j))|>\eps$ if $w\in W\setminus W_j$.
Using \eqref{eq-root-inclusion-for-pj} with $\mu\in U\cap\Lambda$ instead of $\lambda$,
we find that $p_{j,\mu}(s)=0$ has, counting multiplicities,
at most one root in the disk with center $s_0$ and radius $\eps$.
The simplicity of the root $s_0$ of $p_{j,\lambda}$ follows from this observation.
\end{proof}

In the construction of an inverse of the Poisson transform,
we shall divide by the holomorphic function $p:\astarc\to\C$,
\begin{equation}
\label{eq-def-p-lambda}
p(\lambda) = \prod\nolimits _{j=1}^\ell \big(\frac{\intd}{\intd s}\big|_{s=0}\big) p_{j,\lambda}(\rho_j-\lambda_j+s).
\end{equation}
It follows from Corollary~\ref{cor-charexp-Pj} that $p(\lambda)=0$ if and only if
there exist $j\in\{1,\ldots,\ell\}$ and $w\in W\setminus W_j$ such $\lambda(H_j-w\cdot H_j)=0$.
In particular, $p$ is not identically zero.

\begin{remark}
Let $\Phi$ be a local diffeomorphism \eqref{eq-Phi-edge-diffeom} around the edge,
$\Phi^*$ the pullback by $\Phi$, and $\Phi^{-*}=(\Phi^*)^{-1}$.
If $D\in\Diff_0(\OO)$, then $\Phi^{-*}D\Phi^*\in\Diff_0(\OO)$.
The same implication holds with $\Diff_b(\OO)$ instead of $\Diff_0(\OO)$.
These assertion are clear after a straightforward calculation for the vector fields $D=\theta_j$ and $D=\dd{z_k}$.
Furthermore, $\Phi^{-*}P_j(\lambda)\Phi^*$ is of Fuchsian type with respect to the $j$-wall,
and its indicial polynomial equals that of $P_j(\lambda)$.
\end{remark}

\section{Extension of eigenfunctions}
In this section, following the proof of \cite[Theorem 5.12]{KashiwaraOshima77},
we extend joint eigenfunctions from $\OOplus$
across the edge as joint eigenfunctions in $\OO$ with support contained in $\OOplusclosure$.
\begin{lemma}
\label{lemma-Pj-commutation}
Let $D\in\D(X)$.
There exist $S_{jk}\in\Diff_b(\OO)$, $t_j S_{jk}\in\Diff_0(\OO)$, such that
\[ [P_j(\lambda),D] = \sum\nolimits _{k=1}^{m_j} t_j S_{jk}\big(D_{jk}-\chi_\lambda(D_{jk})\big), \]
and $\ord(S_{jk})+\ord(D_{jk})<m_j+\ord(D)$.
\end{lemma}
\begin{proof}
By \eqref{eq-def-Pj} we have
$[P_j(\lambda),D] = \sum\nolimits _{k=1}^{m_j} [\theta_j^{m_j-k},D]\big(D_{jk}-\chi_\lambda(D_{jk})\big)$.
Using $D\in\Diff_0(\OO)$ and \eqref{eq-D-Diffzero-local}, the commutators $[\theta_j^i,D]$ are seen to
belong to $\Diff_0(\OO)\cap t_j\Diff_b(\OO)$.
The orders are $<i+\ord(D)$.
\end{proof}
We deal with solutions in $\OO$ of $\sigma$-shifted eigensystems
\begin{equation}
\label{eq-joint-eigensys-shifted}
D^\sigma \tilde u=\chi_{\lambda}(D)\tilde u,\quad D\in\D(X),
\end{equation}
where $\sigma=\rho-\lambda$ is a characteristic exponent of $\D(X)$.

\begin{proposition}
\label{prop-extension-of-solns}
Let $\lambda\in\C^\ell$.
Assume that, for $w\in W$ and $j=1,\ldots,\ell$,
$(\lambda-w\cdot\lambda)(H_j)$ is not a negative integer.
Set $\sigma=\rho-\lambda$.
Let $u\in\Eigen_\lambda(X)$.
There is a unique $\tilde u\in\Hfcn(\OO)$ such that
\eqref{eq-joint-eigensys-shifted} holds in $\OO$,
$\tilde u|_{\OOplus}=t^{-\sigma} u$, and $\supp(\tilde u)\subseteq \OOplusclosure$.
\end{proposition}
\begin{proof}
By Corollary~\ref{cor-charexp-Pj} and the assumption,
\begin{equation}
\label{eq-non-integer-char-exponents-Pjlambdasigma}
p_{j,\lambda}(\sigma_j-n)\neq 0\quad\text{holds for $j=1,\ldots,\ell$ and $n=1,2,3,\ldots$.}
\end{equation}
Set $O_j=\tilde O\cap \{t_{j+1}>0,\ldots, t_\ell>0\}$, $O_0=O$.

There is a unique $u_1\in\Hfcn(O_1)$ satisfying
\[ P_1(\lambda)^\sigma u_1=0, \quad u_1|_{O_0}=t^{-\sigma}u, \quad \supp(u_1)\subseteq\{t_1\geq 0\}. \]
In fact, by the flabbyness of the sheaf $\Hfcn$, there exists a supported extension $v_1\in\Hfcn(O_1)$
of $t^{-\sigma}u$, that is $v_1|_{O_0}=t^{-\sigma}u$ and $\supp(v_1)\subseteq\{t_1\geq 0\}$ hold.
By \eqref{eq-non-integer-char-exponents-Pjlambdasigma}
and Corollary~\ref{cor-fuchsian-hyperfunc-soln} there is a unique $w_1\in\Hfcn(O_1)$, supported
in the hypersurface $\{t_1=0\}$, which solves $P_1(\lambda)^\sigma w_1= P_1(\lambda)^\sigma v_1$.
Now set $u_1=v_1-w_1$.

We claim that $D^{\sigma} u_1=\chi_\lambda(D)u_1$ holds in $O_1$ for every $D\in\D(X)$.
It suffices to prove this for generators $D_1,\ldots,D_N$ of $\D(X)$.
We assume that the $D_{1k}$'s in \eqref{eq-def-Pj} are among the generators.
By Lemma~\ref{lemma-Pj-commutation} we have
$[P_1(\lambda),D_n-\chi_\lambda(D_n)] = \sum\nolimits _{k} t_1 S_{nk}\big(D_{k}-\chi_\lambda(D_{k})\big)$
with $S_{nk}\in\Diff_b(\OO)$, $\ord(S_{nk})+\ord(D_{k})<m_1+\ord(D_n)$.
Thus
\begin{equation}
\label{eq-FuchsionOp-commutator}
[P_1(\lambda)^\sigma,D_n^\sigma-\chi_\lambda(D_n)] = \sum\nolimits _{k} t_1 S_{nk}^\sigma\big(D_{k}^\sigma-\chi(D_{k})\big).
\end{equation}
Denote by $I_N$ the $N\times N$-unit matrix.
Define the $N\times N$-matrices $P=P_1(\lambda)^\sigma I_N$, $S=(S_{nk}^\sigma)$,
and the column vector $Q$ with $k$-th component $D_k^\sigma-\chi_\lambda(D_k)$.
Then formula~\eqref{eq-FuchsionOp-commutator} reads
\[(P-t_1S)Q=QP_1(\lambda)^\sigma.\]
It follows that $(P-t_1S)Qu_1=0$.
Moreover, $\supp(Qu_1)\subseteq \{t_1=0\}$.
So $P-t_1S$ is a Fuchsian type system with (matrix-valued) indicial polynomial $p_{1,\lambda}(\sigma_1+s)I_N$.
Recall \eqref{eq-non-integer-char-exponents-Pjlambdasigma}.
In view of Remark~\ref{remark-Fuchsian-general-NbyN}, Corollary~\ref{cor-fuchsian-hyperfunc-soln}
applies to show that solutions of $(P-t_1S)v=0$ which are supported in $t_1 = 0$ must vanish.
Thus $Qu_1=0$, proving the claim.

Next, arguing as before, we find a unique $u_2\in\Hfcn(O_2)$ such that 
$P_2(\lambda)^{\sigma} u_2=0$, $u_2|_{O_1}=u_1$, $\supp(u_2)\subseteq \{t_2\geq 0\}$,
and we then deduce that $D^{\sigma} u_2=\chi_\lambda(D)u_2$ holds in $O_2$ for every $D\in\D(X)$.
Continuing in this way, we obtain $\tilde u=u_\ell$.
The uniqueness of the extension $\tilde u$ follows because, for each $j$,
$\tilde u|_{O_j}$ is uniquely determined by $\tilde u|_{O_{j-1}}$.
\end{proof}
If \eqref{eq-joint-eigensys-shifted} holds and $\supp(\tilde u)\subseteq \OOplusclosure\setminus\OOplus$,
then, under the assumptions on $\lambda$ in the proposition, $\tilde u=0$.
This is seen by inspection of the preceding proof.

Next we define an extension operator between Sobolev spaces of negative order
which maps eigenfunctions into $\sigma$-shifted eigenfunctions on $\OO$.
\begin{proposition}
\label{prop-extension-of-solns-Sobolev}
Let $\lambda\in\C^\ell$.
Assume that, for $w\in W$ and $j=1,\ldots,\ell$,
$(\lambda-w\cdot\lambda)(H_j)$ is not a negative integer.
Set $\sigma=\rho-\lambda$.
Let $0\leq m\in\Z$.
There exists a bounded linear operator $V:\Sob{-m}(\OOplus)\to\Sob{-m'}(\OO)$, $m'$ some integer $\geq m$,
such that $\supp(Vu)\subseteq\OOplusclosure$, $(Vu)|_{\OOplus}=t^{-\sigma}u$, and,
if $u\in\Eigen_\lambda(X)$, then $Vu$ satisfies
$(D^\sigma-\chi_\lambda(D))Vu=0$ for all $D\in\D(X)$.
Moreover, $V$ depends holomorphically on $\lambda$,
provided $\lambda$ ranges in a bounded open set in which the assumption holds.
\end{proposition}
\begin{proof}
Notice that $\lambda$ and $\sigma$ satisfy the assumptions of Proposition~\ref{prop-extension-of-solns},
hence also \eqref{eq-non-integer-char-exponents-Pjlambdasigma}.

Define $O_j$ as in the proof of Proposition~\ref{prop-extension-of-solns}.
Denote by $\bar O_{j-1}$ the closure of of $O_{j-1}$ in $O_{j}$, and $N_j=\bar O_{j-1}\setminus O_{j-1}$.
If $A$ is a closed subset of $O_j$,
then $\Sob{-k}_{A}(O_{j})$ denotes the closed subspace of $\Sob{-k}(O_{j})$
which consists of the elements supported in $A$.
There exist extension operators
\[ E_n:\Sob{-n}(O_{j-1})\to \Sob{-2n}_{\bar O_{j-1}}(O_{j}), \]
see \cite[Appendix A]{SHJHAP17resonscattpoles}.
Assume $m_j\leq m$.
The commutator $P_j(\lambda)^\sigma E_m-E_{2m}P_j(\lambda)^\sigma$ maps $\Sob{-m}(O_{j-1})$
into $\Sob{-4m}_{N_j}(O_j)$.
By Proposition~\ref{prop-Fuchsian-in-Sobolev} there exist an integer $m'\geq 4m$ and an operator
$S(\lambda): \Sob{-4m}_{N_j}(O_j)\to \Sob{-m'}_{N_j}(O_j)$ such that $P_j(\lambda)^\sigma S(\lambda)$ is the identity mapping.
Define the operator $V_j:\Sob{-m}(O_{j-1})\to\Sob{-m'}(O_j)$ by
\[ V_j(\lambda) =E_m -S(\lambda) \big(P_j(\lambda)^\sigma E_m-E_{2m}P_j(\lambda)^\sigma\big). \]
Let $u_{j-1}\in \Sob{-m}(O_{j-1})$.
Put $u_j=V_j(\lambda)u_{j-1}\in \Sob{-m'}(O_j)$.
If $\supp u_{j-1}\subseteq \OOplusclosure$ then $\supp u_{j}\subseteq \OOplusclosure$.
Assume that $u_{j-1}$ solves the $\sigma$-shifted joint eigensystem \eqref{eq-joint-eigensys-shifted} in $O_{j-1}$.
Then
\[ P_j(\lambda)^\sigma  u_j=
 P_j(\lambda)^\sigma  V_j(\lambda) u_{j-1} = E_{2m}P_j(\lambda)^\sigma u_{j-1} =0.
\]
Arguing as in the proof of Proposition~\ref{prop-extension-of-solns}, we find that $u_j$
is the unique solution of \eqref{eq-joint-eigensys-shifted} in $O_j$
which restricts to $u_{j-1}$ in $O_{j-1}$ and has its support in $\bar O_{j-1}$.
Summarizing, we see that $V=V_\ell(\lambda) \cdots V_1(\lambda)$ has the desired properties.
\end{proof}

\section{Boundary values}
For generic $\lambda\in\astarc$, we define the boundary value $\bv_\sigma u$ of $u\in\Eigen_\lambda(X)$
with respect to the characteristic exponent $\sigma=\rho-\lambda$.

Fix $\lambda\in\C^\ell=\astarc$.
Put $P_j=P_j(\lambda)$.
Define $\omega_j=(1,\ldots,1,0,\ldots,0)\in\Z^\ell$ such that the first zero occurs in the $j$-th component.
In $\OOplus$ we have
\begin{equation}
\label{eq-prod-Pj-defined}
P:=(t_1\ldots t_\ell)^{-1} P_\ell^\sigma \ldots P_1^\sigma
   = (t_{\ell}^{-1} P_\ell^{\sigma+\omega_\ell}) \cdots (t_{1}^{-1} P_1^{\sigma+\omega_1}).
\end{equation}
Since $\sigma$ is a characteristic exponent, $\ipol_j(P_j^{\sigma+\omega_j})(0)=\ipol_j(P_j)(\sigma_j)=0$ holds.
Therefore $t_j^{-1} P_j^{\sigma+\omega_j}$, and thus also $P$, uniquely extends to a differential operator in $\OO$.

Put $N_j=B\times]-1,1[^{\ell-j}$.
Identify $N_j$ with the subset of $\OO$ given by $t_1=\ldots=t_j=0$.
We say that $T\in\Diff(\OO)$ is tangent to $N_j$ if $w|_{N_j}=0$ implies $(Tw)|_{N_j}=0$.
The tangency of $T$ to $N_j$ is equivalent to $T$ being a finite sum of terms
$a(z,t)\big(\prod_{k\leq j}\theta_{k}^{\beta_k}\big)\big(\prod_{k> j}\partial_{t_k}^{\gamma_k}\big)\partial_z^\kappa$,
where $a$ is analytic in $\OO$.
If $T\in\Diff(\OO)$ is tangent to $N_j$, then we define $T|_j\in\Diff(N_j)$ by
\[ (T|_j v)\otimes\delta(t_1,\ldots,t_j) = T^{\omega_{j+1}} (v\otimes\delta(t_1,\ldots,t_j)), \]
$\delta$ the Dirac function.
Observe that $\theta_{k}^{\omega_{j+1}}=\theta_{k}+1$ if $k\leq j$ and $=\theta_{k}$ if $k>j$.
Recall $(\theta_{k}+1)\delta(t_k)=0$.
Therefore $T|_j$ arises from $T$ by
freezing coefficients at $N_j$ and discarding all terms having some $\beta_k>0$.
So, the product \eqref{eq-prod-Pj-defined} corresponds to successive reduction to boundary faces;
compare \cite[(3.26)]{Oshima83bdryval}.

The following lemma applies to $v_0=\tilde u$ where $\tilde u$ is the extension of
$t^{-\sigma}u$, $u\in\Eigen_\lambda(X)$, obtained in Proposition~\ref{prop-extension-of-solns}.

\begin{lemma}
\label{lemma-prod-Pj-applied}
Let $v_0\in\Hfcn(\OO)$, $\supp(v_0)\subset \OOplusclosure$,
such that $D^\sigma v_0=\chi_\lambda(D)v_0$ for all $D\in\D(X)$.
There exist $v_j\in\Hfcn(N_j)$ such that
\begin{equation}
\label{eq-recursive-defn-of-vj}
(t_{j}^{-1}P_{j}^{\sigma})|_{j-1} v_{j-1}=v_{j}\otimes \delta(t_{j})
\end{equation}
is satisfied for $j=1,\ldots,\ell$.
In particular, setting $u_\sigma=v_\ell\in\Hfcn(B)$, there holds
\begin{equation}
\label{eq-prod-Pj-applied}
Pv_0 = u_\sigma\otimes\delta(t).
\end{equation}
\end{lemma}
\begin{proof}
Set $w_0=v_0$ and $w_j=t_j^{-1}P_j^{\sigma+\omega_j}w_{j-1}$ for $j=1,2,\ldots,\ell$.
Thus $w_\ell=Pv_0$.
First we prove, by induction over $j$, that the following is true:
\begin{equation}
\label{eq-diffeq-wj}
(D^{\sigma+\omega_{j+1}}-\chi_\lambda(D))w_j=0\quad\text{for all $D\in\D(X)$.}
\end{equation}
For $j=0$ this follows from our assumption on $v_0$.
Let $D\in\D(X)$.
By Lemma~\ref{lemma-Pj-commutation} there exist $S_{jk}\in\Diff_b(\OO)$ such that
\[ [P_j, D-\chi_\lambda(D)] = t_j \sum\nolimits _{k=1}^{m_j} S_{jk} (D_{jk}-\chi_\lambda(D_{jk})) \]
holds.
Conjugating with $t^{-\sigma-\omega_j}$, we get
\begin{align*}
(D^{\sigma+\omega_{j+1}} -\chi_\lambda(D))(t_j^{-1}P_j^{\sigma+\omega_j})
   &= (t_j^{-1}P_j^{\sigma+\omega_j})(D^{\sigma+\omega_j}-\chi_\lambda(D)) \\
   &\phantom{==}   -\sum\nolimits _{k=1}^{m_j} S_{jk}^{\sigma+\omega_j} (D_{jk}^{\sigma+\omega_j}-\chi_\lambda(D_{jk})).
\end{align*}
Applying this operator to $w_{j-1}$ completes the inductive step of \eqref{eq-diffeq-wj}.

We claim that there exist $v_j\in\Hfcn(N_j)$ which satisfy \eqref{eq-recursive-defn-of-vj} and
\begin{equation}
\label{eq-wj-vj-tensor-delta}
w_j(z,t) = v_j(z,t_{j+1},\ldots,t_\ell)\otimes\delta(t_1,\ldots,t_j).
\end{equation}
Assume that \eqref{eq-wj-vj-tensor-delta} holds with $j$ replaced by $j-1$.
Since $t_j^{-1}P_j^{\sigma}$ is tangent to $N_{j-1}$, we get
\[
w_j(z,t) = t_j^{-1} P_j^{\sigma+\omega_j}w_{j-1}(z,t)
 = \big((t_j^{-1}P_{j}^{\sigma})|_{j-1} v_{j-1}\big)(z,t_j,\ldots,t_\ell)\otimes\delta(t_1,\ldots,t_{j-1}).
\]
By \eqref{eq-def-Pj} and \eqref{eq-diffeq-wj}, $P_j^{\sigma+\omega_j}w_{j-1}=0$.
Hence $t_j(t_j^{-1}P_{j}^{\sigma})|_{j-1} v_{j-1}=0$,
implying \eqref{eq-recursive-defn-of-vj}.
This completes the proof of \eqref{eq-wj-vj-tensor-delta} and the lemma.
\end{proof}

Recall the definitions \eqref{eq-def-p-lambda} and \eqref{eq-prod-Pj-defined}.
The next lemma shows that $p(\lambda)^{-1}P$ applied to the extension $\tilde u$ recovers
the leading term of an asymptotic expansion of $u$ at  $t=0$.
\begin{lemma}
\label{lemma-bv-continuous-extension}
Let $u\in \Eigen_\lambda(X)$ such that $t^{-\sigma}u$, $\sigma=\rho-\lambda$, extends to a continuous function in $\OOplusclosure$.
Define $v_0\in\Lloc^1(\OO)$ by $v_0|_{\OOplus}=t^{-\sigma}u$ and $\supp(v_0)\subset \OOplusclosure$.
Then $D^\sigma v_0=\chi_\lambda(D)v_0$ holds for all $D\in\D(X)$,
and $u_\sigma$ in \eqref{eq-prod-Pj-applied} is given by
\begin{equation}
\label{eq-recover-contin-bv}
u_\sigma= p(\lambda) v_0|_{t=0+}.
\end{equation}
\end{lemma}
\begin{proof}
Fix $h\in\Cinfty(\R)$, $h(t)=0$ if $t<1/2$, and $h(t)=1$ if $t>1$.
Put $\gamma(t)=h(t_1)\cdots h(t_\ell)$.
If $\psi\in\Cinfty(\R^\ell)$ then we denote by $\psi(t/\eps)$, $0<\eps<1$, the family of operators
on $\Dprime(\OO)$ given by multiplication with the functions $(z,t)\mapsto \psi(t/\eps)$, $z\in B$.
Notice the commutator formula $[\theta_j,\psi(t/\eps)]=(\theta_j\psi)(t/\eps)$.
If $T\in\Diff_b(\OO)$, then, moving $\eps$-dependent factors to right, we can write
$[T,\gamma(t/\eps)]$ as a finite sum of terms $\tilde T \psi(t/\eps)$, $\tilde T\in\Diff_b(\OO)$.
Here $\supp(\psi)\subseteq\supp(\gamma)$, and $t_j\leq 1$ holds on $\supp(\psi)$ for at least one $j$.
Therefore $\lim_{\eps\to 0+} \psi(t/\eps)f=0$ in $\Lloc^1(\OO)$ for every $f\in\Lloc^1(\OO)$.
It follows that
\[ \lim\nolimits _{\eps\to 0+} [T,\gamma(t/\eps)] f=0\quad\text{in $\Dprime(\OO)$ if $f\in\Lloc^1(\OO)$.} \]
We apply this, and the fact that $v_0=\lim_{\eps\to 0+} \gamma(t/\eps) v_0$ in $\Lloc^1(\OO)$, to get
\[ (D^\sigma -\chi_\lambda(D))v_0 =\lim\nolimits _{\eps\to 0+} [D^\sigma -\chi_\lambda(D),\gamma(t/\eps)]v_0=0. \]
This proves the first assertion of the lemma.

Define $v_1, v_2, \ldots, v_\ell$ by \eqref{eq-recursive-defn-of-vj}.
Define $q_j$ by $sq_j(s)=p_j(s+\sigma_j)$, $p_j=\ipol_j(P_j)$.
Then
\[ t_j^{-1}P_j^{\sigma+\omega_j}=\partial_{t_j}q_j(\theta_j)+R_j =q_j(\theta_j+1)\partial_{t_j}+R_j,  \]
where $R_j\in\Diff_b(\OO)$.
The recursive step \eqref{eq-recursive-defn-of-vj} reads
\[ \big(q_j(\theta_j+1)\partial_{t_j}+R_j\big) v_{j-1}=v_j\otimes\delta(t_j).  \]
Here, to simplify notation, we denote the restricted operator $R_j|_{j-1}$ also by $R_j$.
Note that $R_j$ belongs to the algebra generated by the vector fields
$\theta_{j},\ldots,\theta_\ell, \partial_{z_k}$, where $(z_k)$ are local coordinates of $B$.
To compute $v_j$ we procede by induction over $j$.
Assume that $v_{j-1}(b,t_j,\ldots,t_\ell)$ is a locally integrable function on $N_{j-1}$
which is continuous when restricted to the closed corner domain defined by $t_{j}\geq 0$, \dots, $t_\ell \geq 0$,
and which vanishes whenever $t_i<0$ for some $i\geq j$.
Put $\gamma_j(t)=h(t_j)\cdots h(t_\ell)$.
Then
\begin{equation}
\label{eq-lim-comm-vj}
\lim\nolimits _{\eps\to 0+} \big[q_j(\theta_j+1)\partial_{t_j}+R_j,\gamma_j(t/\eps)\big] v_{j-1}=v_j\otimes\delta(t_j)
\end{equation}
in $\Dprime(N_{j-1})$.
We have $[q_j(\theta_j+1),h(t_j/\eps)]=\phi_j(t_j/\eps)$
with $\phi_j\in\Cinfty(\R)$, $\supp(\phi_j)\subset[1/2,1]$.
The commutator in \eqref{eq-lim-comm-vj} equals
\[ \gamma_{j+1}(t/\eps) \big(q_j(\theta_j+1) h'(t_j/\eps)/\eps -\phi_j'(t_j/\eps)/\eps +\partial_{t_j}\phi(t_j/\eps)\big) \]
plus a finite sum of terms of the form $R\psi(t/\eps)$.
Here $R\in\Diff(N_{j-1})$, and the support of the multiplier function $\psi$
is contained in $\supp(\gamma_j)\cap\{0\leq t_n\leq 1\}$ for some $n\in\{j,\ldots,\ell\}$.
This implies that $R\psi(t/\eps)v_{j-1}$ converges to zero in $\Dprime(N_{j-1})$ as $\eps\to 0+$.
We have
\[ \lim\nolimits _{\eps\to 0+} \gamma_{j+1}(t/\eps) \eps^{-1}\phi_j'(t_j/\eps) =0\]
in the space of Radon measures.
Hence \eqref{eq-lim-comm-vj} reduces to
\[ \lim\nolimits _{\eps\to 0+} \gamma_{j+1}(t/\eps) q_j(\theta_j+1)\eps^{-1} h'(t_j/\eps) v_{j-1}=v_j\otimes\delta(t_j). \]
In view of $\lim\nolimits _{\eps\to 0+} h'(t_j/\eps)/\eps =\delta(t_j)$
and $q_j(\theta_j+1)\delta(t_j)=q_j(0)\delta(t_j)$
we have shown that $v_j=q_j(0) v_{j-1}|_{t_j=0}$.
As $u_\sigma=v_\ell$ and $q_j(0)=p_j'(0)$ formula \eqref{eq-recover-contin-bv} follows.
\end{proof}

Now, we define the boundary value operator $\bv_{\rho-\lambda}$ for generic $\lambda$.
\begin{proposition}
\label{prop-defn-bv}
Let $\lambda\in\astarc$ such that $p(\lambda)\neq 0$.
Assume that $\lambda(H_j-w\cdot H_j)$ is not a negative integer if $j=1,\ldots,\ell$ and $w\in W$.
Put $\sigma=\rho-\lambda$.
Then
\begin{equation}
\label{def-bv-sigma}
\bv_{\sigma}:\Eigen_\lambda(X)\to\Hfcn(B)=\Aprime(B), \quad
u\mapsto p(\lambda)^{-1} \int P \tilde u \intd t
\end{equation}
defines a bounded linear operator.
Here $\tilde u$ is the unique extension of $t^{-\sigma}u$ obtained in Proposition~\ref{prop-extension-of-solns};
$p(\lambda)$ and $P$ are defined in \eqref{eq-def-p-lambda} and \eqref{eq-prod-Pj-defined}, respectively.
The integral is the pushforward of hyperfunctions under the canonical projection $\OO\to B$.
Furthermore, for every non-negative integer $m$ there exists $\bv_{\sigma,m}:\Sob{-m}(\OOplus)\to\Dprime(B)$,
linear, bounded, and holomorphic in $\lambda$,
such that $\bv_\sigma u=\bv_{\sigma,m}u$ holds for $u\in\Eigen_\lambda(X)\cap\Sob{-m}(\OOplus)$.
In particular, $\bv_\sigma$ restricts to a bounded linear operator from $\mEigen_\lambda(X)$ into $\Dprime(B)$.
\end{proposition}
\begin{proof}
Using Proposition~\ref{prop-extension-of-solns} and Lemma~\ref{lemma-prod-Pj-applied},
we see that \eqref{def-bv-sigma} defines a linear map.
Observe that $\int P \tilde u \intd t$ equals $u_\sigma$ in \eqref{eq-prod-Pj-applied} with $v_0=\tilde u$.
The graph of $\bv_{\sigma}$ is closed because $\bv_\sigma$ is a composition of bounded operators with an extension operator.
The graph of the latter is closed by the uniqueness of the extensions $\tilde u$.
So, $\bv_\sigma$ is continuous by the closed graph theorem for Fr\'echet spaces.

Using the linear operator $V$ from Proposition~\ref{prop-extension-of-solns-Sobolev}, we define
$\bv_{\sigma,m}u=p(\lambda)^{-1}\int PVu\intd t$ if $u\in\Sob{-m}(\OOplus)$.
Observe that this expression equals $\bv_\sigma u$ if $u\in\Eigen_\lambda(X)$.
The remaining assertions follow.
\end{proof}

Boundary values are sections of the $G$-homogeneous line bundle $(N^*B)^\sigma\to B$.
\begin{proposition}
\label{prop-bv-defn-equivar}
Let $\lambda$ as in Proposition~\ref{prop-defn-bv}, $\sigma=\rho-\lambda$.
Then
\[\bv_{\sigma}:\Eigen_\lambda(X)\to\Hfcn(B;(N^*B)^{\sigma})\]
is $G$-equivariant.
\end{proposition}
\begin{proof}
Let $u\in\Eigen_\lambda(X)$,
and $\tilde u$ the unique extension of $t^{-\sigma}u$ obtained in Proposition~\ref{prop-extension-of-solns}.
Let $\Phi$ as in \eqref{eq-Phi-edge-diffeom}.
We derive a formula for $\bv_\sigma u$ with respect to the local
trivialization of $(N^*B)^\sigma$ given by $\Phi$.

Put $v(z',t')=u(z,t)$, that is $v=\Phi^{-*}u$ in $\OOplus'$, where
$\Phi^*$ denotes the pullback by $\Phi$, and $\Phi^{-*}=(\Phi^*)^{-1}$.
Now $\tilde v= \Phi^{-*} \tau^{-\sigma}\tilde u$ is the unique extension of $(t')^{-\sigma}v$,
supported in $\OOplusclosure'$, which solves 
$(t')^{-\sigma}\big(\Phi^{-*} (D-\chi_\lambda)\Phi^{*}\big) (t')^{\sigma} \tilde v=0$
for every $D\in\D(X)$.
Note here that the argument which proves $t^{-\sigma}\Diff_0(\OO)t^\sigma\subseteq\Diff_0(\OO)$
also applies to the present situation with primed variables.
Define the differential operators $P_j'=\Phi^{-*}P_j\Phi^*$, where, as before, $P_j=P_j(\lambda)$, and
\[ P' =(t_1'\ldots t_\ell')^{-1} (t')^{-\sigma} (P_\ell' \ldots P_1')(t')^{\sigma}.  \]
Observe that $P'$ uniquely extends to a neighbourhood of the edge by the formula
\begin{equation}
\label{eq-P-and-Pprime}
\Phi^* P' \Phi^{-*} =(\tau_1\cdots\tau_\ell)^{-1} \tau^{-\sigma} P \tau^\sigma.
\end{equation}
Here $P$ is as in \eqref{eq-prod-Pj-defined}.
We claim that
\[ \bv_{\sigma,\Phi} u:= p(\lambda)^{-1} \int P' \tilde v \intd t' \, (\intd t')^\sigma \]
exists and represents $\bv_\sigma u$ in the trivialization by $\Phi$.
To see this, write $P\tilde u= b(z)\delta(t)$ using \eqref{eq-prod-Pj-applied}.
Set $\Phi_0=\Phi|_{t=0}:z\mapsto z'(z,0)$.
Put $b'=\Phi_0^{-*}b$, that is $b'(z')=b(z)$.
By \eqref{eq-P-and-Pprime}, we have
$\Phi^* P'\tilde v = (\tau_1\ldots\tau_\ell)^{-1} \tau^{-\sigma}P\tilde u$.
Therefore,
\begin{align*}
\int P' \tilde v \intd t' \, (\intd t')^\sigma 
  &= \int \Phi^{-*} \big((\tau_1\ldots\tau_\ell)^{-1} \tau^{-\sigma}b(z)\delta(t)\big)\intd t'\, (\intd t')^\sigma \\
  &= \int \Phi^{-*} \big(\tau^{-\sigma}b(z)\big)\delta(t')\intd t'\, (\intd t')^\sigma.
\end{align*}
Here we used
$\Phi^*\delta(t')= (\tau_1\ldots\tau_\ell)^{-1} \delta(t)$,
which follows from the homogeneity properties of the Dirac function.
So, in view of \eqref{eq-trafo-NstarBsigma}, we have
\[
p(\lambda) \bv_{\sigma,\Phi} u= b'(z') \Phi_0^{-*}(\tau^{-\sigma}) (\intd t')^\sigma =b(z) (\intd t)^\sigma,
\]
which proves the claim.
(Compare the proof of assertion 1) of \cite[Theorem 3.7]{Oshima83bdryval}.)

Let $g\in G$, and define $\Phi_g$ as in \eqref{eq-Phi-g}.
So, when restricted to $\Eigen_\lambda(X)$, $\Phi_g^*$ is the left translation $L_g$ by $g$.
The diffeomorphism $\Phi_g$ defines a coordinate change $(z,t)\mapsto(z',t')$ as above,
which satisfies
\[ \Phi_{g,0}:kM\mapsto\kappa(g^{-1}k)M, \quad \tau_j(z,0) = e^{\alpha_j(A(gK,kM))}, \quad z=kM\in B. \]
Hence $(\intd t')^\sigma= e^{\sigma(A(gK,kM))}(\intd t)^\sigma$.
The computation in the previous paragraph shows that, for $u\in\Eigen_\lambda(X)$,
$\bv_\sigma(L_g u) = b(z)(\intd t)^\sigma$ holds if
$\bv_\sigma(u) = b'(z') (\intd t')^\sigma$.
This implies the $G$-equivariance of $\bv_\sigma$.
\end{proof}

\section{The Poisson isomorphism}

The Poisson transform $\Poisson_\lambda$ is a $G$-isomorphism iff $\lambda$ is
not a zero of the Harish-Chandra $\efcn$-function, $\efcn_\lambda\neq 0$.
This is the main result of \cite{Kashiwara78eigen}, and it affirms Helgason's conjecture.
With a stronger assumption on $\lambda$, we prove that the boundary value map defined
in the previous section is, up to a scalar factor, the inverse of the Poisson transform.
\begin{theorem}
\label{theorem-kkmoot}
Let $\lambda\in\astarc$ such that
\begin{inparaenum}[(i)]
\item\label{lambda-simple-ind-roots}
$\lambda(H_j-w\cdot H_j)\neq 0$ whenever $w\cdot H_j\neq H_j$,
and
\item\label{lambda-not-negative-int}
$\lambda(H_j-w\cdot H_j)$ is never a negative integer.
\end{inparaenum}
Then
\begin{equation}
\label{eq-kkmoot}
\bv_{\rho-\lambda} \Poisson_\lambda f= \cfcn(\lambda) f \quad\text{for all $f\in\Hfcn(B;(N^*B)^{\rho-\lambda})$.}
\end{equation}
Moreover, if $\cfcn(\lambda)\neq 0$, then $\bv_{\rho-\lambda}$ is bijective,
and $\Poisson_\lambda:\Hfcn(B;(N^*B)^{\rho-\lambda})\to \Eigen_\lambda(X)$
and the restriction $\Poisson_\lambda:\Dprime(B;(N^*B)^{\rho-\lambda})\to \mEigen_\lambda(X)$ 
are $G$-isomorphisms.
\end{theorem}
We remark that assumption \eqref{lambda-not-negative-int}
implies assumption $(A')$ of \cite{Kashiwara78eigen}, which in turn implies $\efcn_\lambda\neq 0$.

The following Fatou--type theorem is well-known.
See \cite[Ch.~\RN 2, Theorem 3.16]{Helgason94GASS}, and, for the stronger assertion made below,
see \cite[Theorem~5.1.4]{Schlichtkrull84hyperf}.
\begin{proposition}
\label{prop-Helgason-Randwert}
Let $\lambda\in\astarc$, $\RE \lambda\in\lie{a}_+^*$.
Let $f\in C(B)$.
Then $v(b,t) = t^{-\rho+\lambda}(\Poisson_{\lambda} f)(b,t)$ extends from $\OOplus$ to a
continuous function $v:\OOplusclosure\to\C$ such that
$v(b,0) = \cfcn_\lambda f(b)$, $b\in B$.
\end{proposition}

\begin{proof}[Proof of Theorem~\ref{theorem-kkmoot}]
Propositions \ref{prop-defn-bv} and \ref{prop-bv-defn-equivar}
show that $\bv_{\sigma}$, $\sigma=\rho-\lambda$, is defined and $G$-equivariant.
Lemma~\ref{lemma-bv-continuous-extension} and Proposition~\ref{prop-Helgason-Randwert} imply \eqref{eq-kkmoot}
when $f$ is continuous.
Since the space of continuous functions is dense in the space of analytic functionals,
\eqref{eq-kkmoot} holds in general.
The $G$-equivariance of $\bv_\sigma$ implies its injectivity, hence its bijectivity;
see \cite[page 22f]{Kashiwara78eigen}.
Let us recall the argument.
Suppose there is $0\neq u\in\Eigen_\lambda(X)$ such that $\bv_\sigma(u)=0$.
By the $G$-equivariance of $\bv_\sigma$ we can assume $u(o)=1$.
Then $\int_K \pi(k)u\intd k=\sfcn_\lambda$, where $\pi(k)$ denotes left translation by $k$.
So,
\[ \bv_\sigma \sfcn_\lambda =\int_K \bv_\sigma \pi(k) u\intd k =\int_K \pi(k) \bv_\sigma u\intd k =0.  \]
But this contradicts $0\neq\cfcn_\lambda = \bv_\sigma \Poisson_\lambda 1 = \bv_\sigma \sfcn_\lambda$.
The remaining assertions are clear by the mapping properties of $\Poisson_\lambda$
and $\bv_\sigma$ already established.
\end{proof}

\appendix

\section{Fuchsian type equations}
\label{section-Fuchsian-eqns}

Let $\tau>0$, and let $N$ be an analytic manifold.
Put $M=N\times]-\tau,\tau[$.
Denote the coordinate on the interval by $t$.
Then $t$ is a defining function of the hypersurface $N\times\{0\}\subset M$ which we identify with $N$.
Denote by $M_\pm$ the open subsets of $M$ defined by $\pm t>0$.

For linear differential operators on $M$ we always assume that the coefficients are analytic.
We say that a differential operator $P$ is of Fuchsian type along the hypersurface $N$
iff there exists a polynomial $p\in\C[s]$, called the indicial polynomial of $P$,
which has degree equal to the order of $P$, and is such that $P$ has the form
\begin{equation}
\label{eq-Fuchsian-type}
P=p(\tdt) +t Q(x,t,\dd{x},\tdt).
\end{equation}
Here and below $x=(x_i)$ are local coordinates on $N$.
The differential operator $Q$ is tangent to $N$.
The formula $p(s)=(t^{-s}Pt^s)|_{t=0+}$ recovers the indicial polynomial from the operator.
Classical homogeneous distributions, pulled back to $M$ by $t:M\to\R$, are mapped by $P$ into their multiples
\[ p(\tdt)t_+^s=p(s)t_+^s, \quad p(\tdt)\delta^{(k)}(t)=p(-k-1)\delta^{(k)}(t), \]
plus less singular terms caused by the perturbation $tQ$.
The zeros of $p$ are called characteristic exponents of $P$.

If one allows indicial polynomials with coefficients in the space of analytic functions on $N$,
then Fuchsian type means the same as regular singular in a weak sense as defined in \cite[Definition 1.3]{Oshima83bdryval}.
In \cite{BaouendiGoulaouic73}, the term Fuchsian type with weight zero is used.
For the wider class of $b$-differential operators introduced in \cite{Melrose93indexthm}
indicial polynomials take values in differential operators on $N$.

Let $P$ be of Fuchsian type along $N$, $p$ its indicial polynomial.
Then $t^{-\lambda}Pt^\lambda$, $\lambda\in\C$, initially a differential operator in $M_+$,
extends uniquely from $M_+$ to a differential operator on $M$.
Moreover, the extension, also denoted $t^{-\lambda}Pt^\lambda$, is of Fuchsian type,
and its indicial polynomial is $p(\lambda+\cdot)$.
If $p(0)=0$, then $t^{-1}P$ extends uniquely from $M_+$ to a differential operator on $M$.

Following \cite{BaouendiGoulaouic73}, we shall find holomorphic solutions
of Fuchsian type equations $Pu=f$ by the method of successive approximations.
For this purpose we derive estimates of solutions of ordinary Fuchsian equations with constant coefficients.

Let $p(s)=\prod\nolimits _{k=1}^m (s-s_k)$ be a monic polynomial of degree $m$.
Then $p(\tdt)u(t)=f(t)$,
is an $m$-th order ordinary differential equation which is regular singular at $t=0$.
If $\omega$ is a star-shaped open neighbourhood of the origin in $\R$ or $\C$,
and if $m>0$ and $\RE s_k<0$ for $k=1,\ldots,m$, then, for $f\in C(\omega)$, we set
\begin{equation}
\label{eq-Fuchsian-Hpf}
H_p f(t)=\int_{[0,1]^m} \prod\nolimits _{k=1}^m \sigma_k^{-s_k-1} f(\sigma_1\cdots \sigma_m t)\intd(\sigma_1,\ldots,\sigma_m).
\end{equation}
Clearly, $H_pf\in C(\omega)$.
If $f$ is differentiable in $\omega$ then so is $H_p f$.
We set $H_1 f=f$ in the case of the constant polynomial $p=1$.
\begin{lemma}
\label{lemma-Hp-FuchsianODE-inverse}
Let $p$ be a monic polynomial of degree $m>0$ such that, for some $\delta>0$,
the zeros of $p$ are contained in the half-plane $\RE s\leq -\delta$.
Let $\omega$ be a star-shaped open neighbourhood of the origin in $\C$.
Let $n$ be a non-negative integer, and let $f$ be a holomorphic function in $\omega$.
Then
\begin{equation}
\label{eq-Fuchsian-Hp-bound}
\sup_{t\in\omega} |t|^{-n}|H_p f(t)| \leq (n+\delta)^{-m}\sup_{t\in\omega} |t|^{-n}|f(t)|
\end{equation}
holds.
Let $q_1$ and $q_2$ be monic polynomials.
Then
\begin{equation}
\label{eq-Fuchsian-theta-Hp}
q_1(\tdt) H_p f = H_{q_2} f \quad\text{if $p=q_1q_2$.}
\end{equation}
In particular, $p(\tdt)H_p f=f$.
\end{lemma}
\begin{proof}
Write $p(s)=\prod\nolimits _{k=1}^m (s-s_k)$.
Using the assumption, we estimate \eqref{eq-Fuchsian-Hpf}.
We get
\[
|H_p f(t)| \leq \int_{[0,1]^m} \sigma_1^{\delta-1}\cdots\sigma_m^{\delta-1}
  |\sigma_1\cdots \sigma_m t|^n\intd\sigma \cdot \sup_{t\in\omega}|t^{-n}f(t)|,
\]
which implies \eqref{eq-Fuchsian-Hp-bound}.
Observe that
$\tdt f(\sigma_1\cdots \sigma_m t) =\sigma_1 \dd{\sigma_1}f(\sigma_1\cdots \sigma_m t)$
holds.
Performing an integration by parts with respect to $\sigma_1$ we get
\begin{align*}
\tdt H_p f(t)
  &= \int_{[0,1]^m} s_1 \prod\nolimits _{k=1}^m \sigma_k^{-s_k-1} f(\sigma_1\cdots\sigma_m t)\intd \sigma \\
  &\phantom{=} + \int_{[0,1]^{m-1}} \prod\nolimits _{k=2}^m \sigma_k^{-s_k-1}f(\sigma_2\cdots\sigma_m t)\intd(\sigma_2,\ldots,\sigma_m)\\
  &= s_1 H_p f(t) + H_{p/(s-s_1)}f(t).
\end{align*}
This proves \eqref{eq-Fuchsian-theta-Hp} in the case of $q_1(s)=s-s_1$.
The general case is readily deduced from this.
In particular, we have proved the last assertion of the lemma.
\end{proof}

The following result agrees, under the additional assumption that the real parts
of characteristic exponents are negative, with \cite[Proposition 2]{BaouendiGoulaouic73}.
The general case is \cite[Theorem 2.4]{Oshima83bdryval}.
Essentially we follow \cite{BaouendiGoulaouic73}.
For $\rho>0$ we denote by $D_\rho\subset\C$ the open disk with radius $\rho$ and center at the origin.

\begin{proposition}
\label{prop-fuchsian-holom-soln}
Let $P$ be of Fuchsian type.
Assume that no non-negative integer is among the characteristic exponents.
Let $\Omega\subset\C^n$ be open and bounded, $\Omega_1$ an open neighbourhood of $\overline{\Omega}$.
Assume that $P$ extends to a holomorphic differential operator in $D_{\rho_1}\times\Omega_1$.
There exists $0<\rho<\rho_1$ such that the following holds.
If $\omega\subseteq D_\rho$ is open and starshaped with respect to the origin
and if $f$ is holomorphic in $\omega\times\Omega_1$,
then there exists a unique function $u$,
defined and holomorphic in $\omega\times\Omega$, which satisfies $Pu=f$.
\end{proposition}
\begin{proof}
Denote the indicial polynomial of $P$ by $p$.
We can assume that $p$ is monic.
The equation $Pu=f$ holds in formal power series
$u(t,x)=\sum\nolimits _k u_k(x)t^k$ and $f(t,x)=\sum\nolimits _k f_k(x)t^k$
if and only if for every $k=0,1,2,\ldots$ the following holds:
\begin{equation}
\label{eq-fuchsian-formal-soln}
p(k)u_k(x) + \sum\nolimits _{j<k} S_{jk}(x,\dd{x}) u_j(x) = f_k(x).
\end{equation}
Here $S_{jk}$ are differential operators.
The equations \eqref{eq-fuchsian-formal-soln} are uniquely solvable because $p(k)\neq 0$ by assumption.

To pass from formal to convergent power series solutions it suffices to show that,
for some positive integer $k$, the equation $Pt^ku=t^kf$ has a solution $u$.
The characteristic exponents of $t^{-k}Pt^k$ are those of $P$ shifted by $-k$ towards the left half-plane.
We assume that $p(s)=0$ implies $\RE s\leq -2$.
It suffices to prove the existence of holomorphic solutions of $Pu=f$ under this additional assumption.

We set $p_j(s)=\prod_{k\leq j}(s-s_k)$.
So $p_j$ divides  $p=p_m$.
By \eqref{eq-Fuchsian-type}, we can write
$P=p(\tdt) +\sum\nolimits _{j\leq m} tQ_{m-j}p_j(\tdt)$.
Here $Q_{m-j}(t,x,\dd{x})$ are differential operators of orders $\leq m-j$.

Denote by $\Omega_r\subset\C^n$ the union of the open polydisks with radii $r>0$ and centers in $\Omega\subset\C^n$.
Fix $0<R\leq 1$ such that $\Omega_R\subseteq\Omega_1$.
Denote by $E_r$, $0<r<R$, the Banach space of holomorphic functions $v$ in $\Omega_r$
which have finite norm $\|v\|_{E_r}=\sup_{x\in\Omega_r}|v(x)|$.
If $f$ is bounded holomorphic in $\omega\times\Omega_r$ we regard $f$ also
as a holomorphic function $\omega\to E_r$, $t\mapsto f(t)=f(t,\cdot)$.
Set $H=H_{p}$ according to \eqref{eq-Fuchsian-Hpf}.
We regard $H$ as an operator which acts on holomorphic $E_r$-valued functions.

With $\rho$ to be determined let $\omega\subseteq D_\rho$ open and starshaped with respect to the origin.
Let $f$ be holomorphic in $\omega\times\Omega_1$.
Assume that $\|f(t)\|_{E_r}\leq 1$ holds if $0<r<R$ and $t\in \omega$.
We show the existence of $v$ holomorphic in $\omega\times\Omega$ which solves the fixed point equation
\begin{equation}
\label{eq-Fuchsian-fixed-pt}
v= f - \sum\nolimits _{j\leq m} tQ_{m-j}p_j(\tdt) Hv.
\end{equation}
By Lemma~\ref{lemma-Hp-FuchsianODE-inverse}, $u=Hv$ satisfies $p(\tdt)u=v$.
Therefore, \eqref{eq-Fuchsian-fixed-pt} gives $Pu=f$.

Applying Cauchy estimates we get
\[ \|\dd{x_k} v\|_{E_{r_1}} \leq (r_2-r_1)^{-1} \|v\|_{E_{r_2}} \quad \text{if $0<r_1<r_2<R$.} \]
Hence there exists a constant $K\geq 1$ such that
\begin{equation}
\label{eq-Fuchsian-Cmj-estimate}
\|Q_{m-j}v\|_{E_{r_1}} \leq K(r_2-r_1)^{-(m-j)}\|v\|_{E_{r_2}}, \quad 0<r_1<r_2<R,
\end{equation}
uniformly in $t\in D_{\rho_1/2}$.
It follows from Lemma~\ref{lemma-Hp-FuchsianODE-inverse} that
\begin{equation}
\label{eq-Fuchsian-thetaH-estimate}
\sup\nolimits _{t\in\omega} |t|^{-n}\|p_j(\tdt)Hw(t)\|_{E_{r}} \leq
    (n+2)^{-(m-j)} \sup\nolimits _{t\in\omega} |t|^{-n}\|w(t)\|_{E_{r}}
\end{equation}
holds when $0<r<R$, $0\leq n\in\Z$, and $j=0,1,\ldots,m$.

We solve \eqref{eq-Fuchsian-fixed-pt} by successive approximation.
Set $v_{-1}=0$,
$v_{n}= f - \sum\nolimits _{j\leq m} tQ_{m-j}p_j(\tdt) Hv_{n-1}$.
Put $w_n=v_n-v_{n-1}$.
Then $w_0=v_0=f$, and
\[ w_{n+1}= - \sum\nolimits _{j\leq m} tQ_{m-j}p_j(\tdt) Hw_{n}, \quad n>0. \]
We claim that
\begin{equation}
\label{eq-Fuchsian-successive-estimate}
\sup\nolimits _{t\in\omega} |t|^{-n}\|w_n(t)\|_{E_r} \leq K^n (m+1)^n \big(e^{n}/(R-r)^n\big)^m, \quad 0<r<R,
\end{equation}
holds for $n=0,1,2,\ldots$
Recalling our assumptions on $f$, we see that \eqref{eq-Fuchsian-successive-estimate} holds when $n=0$.
Assuming $0<r<r+\delta<R$, and applying
\eqref{eq-Fuchsian-Cmj-estimate} and \eqref{eq-Fuchsian-thetaH-estimate},
we estimate as follows:
\begin{align*}
\sup\nolimits _{t\in\omega} |t|^{-n-1}\|w_{n+1}(t)\|_{E_{r}}
   &\leq K \sum\nolimits _{j\leq m}\delta^{-(m-j)} \sup\nolimits _{t\in\omega} |t|^{-n}\|p_j(\tdt) Hw_{n}(t)\|_{E_{r+\delta}} \\
   &\leq K \sum\nolimits _{j\leq m}((n+2)\delta)^{-(m-j)} \sup\nolimits _{t\in\omega} |t|^{-n}\|w_{n}(t)\|_{E_{r+\delta}}.
\end{align*}
We apply the inductive assumption \eqref{eq-Fuchsian-successive-estimate}, choose $\delta=(R-r)/(n+2)$, and deduce
\[
\sup\nolimits _{t\in\omega} |t|^{-n-1}\|w_{n+1}(t)\|_{E_{r}}
   \leq K^{1+n} (m+1)^n \big(e^{n}(1+1/n)^n/(R-r)^n\big)^m
       \sum\nolimits _{j\leq m} (R-r)^{-(m-j)}.
\]
Since $R\leq 1$, the sum is $\leq (m+1)(R-r)^{-m}$.
It follows that \eqref{eq-Fuchsian-successive-estimate} holds with $n$ replaced by $n+1$, proving the claim.

Fix $\rho\in]0,\rho_1/2[$ such that
\[ \sum\nolimits _n K^n (m+1)^n \big(e^{n}/(R/2)^n\big)^m \rho^n<\infty. \]
It follows from \eqref{eq-Fuchsian-successive-estimate} 
that $v=\sum_{n=0}^\infty w_n$ converges uniformly in $\omega\times\Omega_{R/2}$
to a holomorphic solution of \eqref{eq-Fuchsian-fixed-pt}.
Then $u=Hv$ solves $Pu=f$ in $\omega\times\Omega$.
The uniqueness of $u$ follows because formal power series solutions are unique and $\omega$ is connected.
To remove the boundedness assumption on $f$ replace $\omega$ in the proof by its dilates
$(1-\eps)\omega$ which are relatively compact subsets of $\omega$,
and then, using uniqueness, let $\eps\to 0+$.
\end{proof}

Next we consider hyperfunction solutions of Fuchsian type equations.
Let $K\subseteq M$ compact.
By definition, $\Analytic(K)$ is the inductive limit of the spaces of holomorphic functions
defined in complex open neighbourhoods of $K$.
It is known that $\Analytic(K)$ is a (DFS)-space, and
the strong dual of $\Analytic(K)$ is the (FS)-space $\Aprime(K)$ of analytic functionals carried by $K$.

The following results and their proofs are from \cite[\S 3]{Oshima83bdryval}.

\begin{proposition}
\label{prop-fuchsian-anafunc-soln}
Let $P$ be of Fuchsian type.
Assume that no negative integer is among the characteristic exponents.
Let $K\subseteq N$ compact.
Then $P:\Aprime(K)\to\Aprime(K)$ is an isomorphism.
\end{proposition}
\begin{proof}
The transpose $P'$ of $P$ is of Fuchsian type, and, because of $(\tdt)' = -\tdt -1$,
the indicial polynomial of $P'$ is $s\mapsto p(-s-1)$.
Hence $P'$ satisfies the assumptions of Proposition~\ref{prop-fuchsian-holom-soln},
Therefore, $P':\Analytic(K)\to\Analytic(K)$ is an isomorphism,
and so is its dual map $P$.
\end{proof}

Denote by $\Hfcn(M)$ the space of hyperfunctions on $M$.
The sheaf of hyperfunctions is flabby,
so for every (non-empty) open $U\subset M$ the restriction map $\Hfcn(M)\to\Hfcn(U)$ is surjective.
If $U\subseteq M$ is open and bounded, then $\Hfcn(U)=\Aprime(\bar U)/\Aprime(\partial U)$.
If $v\in\Aprime(K)$, $U\subseteq K$, then $\dot v\in\Hfcn(U)$ denotes the hyperfunction represented by $v$.
Moreover, $\supp(\dot v)=U\cap\supp(v)$.

\begin{corollary}
\label{cor-fuchsian-hyperfunc-soln}
Let $P$ satisfy the assumptions of Proposition~\ref{prop-fuchsian-anafunc-soln}.
Then $P$ maps the subspace of hyperfunctions on $M$ which are supported in $N$ bijectively onto itself.
\end{corollary}
\begin{proof}
(Compare \cite[Theorem 4.5]{KashiwaraOshima77} and \cite[Theorem 3.1]{Oshima83bdryval}.)
Represent $f$ as a locally finite sum $f=\sum_j \dot f_j$, $f_j\in\Aprime(K_j)$, $K_j\subset N$ compact.
By Proposition~\ref{prop-fuchsian-anafunc-soln} there exist $u_j\in\Aprime(K_j)$ such that $Pu_j=f_j$.
Then $u=\sum_j \dot u_j\in\Hfcn(M)$ satisfies $Pu=f$, and $\supp(u)\subset N$ holds.

To prove injectivity let $u\in\Hfcn(M)$, $\supp(u)\subseteq N$, $Pu=0$.
Write $u$ as a locally finite sum $u=\sum_j \dot u_j$, where $u_j\in\Aprime(K_j)$, $K_j\subset N$ compact.
Let $V$ be an open subset of $M$ such that the indices $j$ which satisfy $V\cap K_j\neq\emptyset$ form a finite set $J$.
Put $v=\sum_{j\in J} u_j\in\Aprime(K')$, $K'=\cup_{j\in J} K_j$.
Then $u|_V=\dot v|_V$ in $\Hfcn(V)$.
By Proposition~\ref{prop-fuchsian-anafunc-soln} we find $w\in\Aprime(\supp(Pv))$ such that $Pv=Pw$.
The injectivity of $P$ on $\Aprime(K')$ implies $v=w$.
Therefore, $\supp(v)=\supp(Pv)$, and
\[ \supp(v)\cap V=\supp(Pv)\cap V = \supp(P\dot v|_V) =\supp(Pu|_V)=\emptyset. \]
Hence $u|_V=0$.
Since the open subsets $V$ cover $M$, we have proved $u=0$.
\end{proof}

\begin{remark}
\label{remark-Fuchsian-general-NbyN}
An inspection of the proof of Proposition~\ref{prop-fuchsian-holom-soln} reveals that
the result holds more generally for $n\times n$-systems $Pu=f$, $P=p(\tdt)I_n+tQ$,
where $I_n$ is the $n\times n$ unit matrix and $Q=\big(Q_{ij}(x,t,\dd{x},\tdt)\big)$
an $n\times n$-matrix of scalar differential operators having orders $\leq \deg p$.
The indicial polynomial $p$ is scalar.
As a consequence Proposition~\ref{prop-fuchsian-anafunc-soln}
and Corollary~\ref{cor-fuchsian-hyperfunc-soln} also hold for this more general
class of Fuchsian type operators.
\end{remark}

We denote by $\Sob{-m}_N(M)$ the set of $f\in \Sob{-m}(M)$, $0\leq m\in\Z$, satisfying $\supp(f)\subseteq N$.
This is a closed subspace of the Sobolev space $\Sob{-m}(M)$.

\begin{proposition}
\label{prop-Fuchsian-in-Sobolev}
Let $P$ satisfy the assumptions of Proposition~\ref{prop-fuchsian-anafunc-soln}.
Let $0<m\in\Z$.
There exists $m'\in\Z$, $m\leq m'$, and a bounded linear operator
$S:H_N^{-m}(M)\to H_N^{-m'}(M)$ such that $PS=\id$.
If the coefficients of $P$ depend holomorphically on a parameter $\lambda$,
then $\lambda\mapsto S$ is a holomorphic operator-valued map.
\end{proposition}
\begin{proof}
First, we observe that there are differential operators $Q_{kj}$ on $N$,
having orders not larger than the order of $P$, such that for $v\in\Dprime(N)$
the following identity holds:
\[ P (v\otimes\delta^{(k)}(t))
  = p(-k-1)(v\otimes\delta^{(k)}(t)) +\sum\nolimits _{j<k} (Q_{kj} v) \otimes\delta^{(j)}(t),
\]
Every $f\in \Sob{-m}_N(M)$ is of the form
$f= \sum\nolimits _{k=0}^m f_k\otimes \delta^{(k)}(t)$
with $f_k\in\Sob{-m_1}(N)$, $m_1$ not depending on $f$.
For $v= \sum\nolimits _{k=0}^m v_k\otimes \delta^{(k)}(t)$, $v_k\in\Dprime(N)$,
the equation $Pv=f$ is equivalent to a triangular system for the vector $(v_k)_k$.
Because or our assumption the diagonal elements of the triangular matrix are non-zero.
Therefore, we can solve $Pv=f$ uniquely by a distribution $Sf:=v$, $\supp(v)\subset N$.
Moreover, there exists $m_2$ depending on $m$ and $m_1$ such that $v_k\in\Sob{-m_2}(N)$.
Hence $S$ is defined and has the asserted mapping properties for some $m'$.
If the coefficients of $P$ depend holomorphically on $\lambda$, then so do the coefficients of $Q_{kj}$,
and therefore also $S$ will depend holomorphically on $\lambda$.
\end{proof}

\newcommand{\etalchar}[1]{$^{#1}$}

\end{document}